\documentclass[11pt]{amsart}

\usepackage{amsmath,amssymb,amscd,amsxtra}
\usepackage{xcolor}
\usepackage{graphicx}
\usepackage{mathrsfs}
\usepackage{mathtools}
\usepackage{float}
\usepackage{amsrefs}
\usepackage{enumerate}
\usepackage{dsfont}

\newtheorem{thm}{Theorem}[section]
\newtheorem{prop}[thm]{Proposition}

\newtheorem{cor}[thm]{Corollary}

\theoremstyle{definition}
\newtheorem{definition}[thm]{Definition}

\theoremstyle{remark}
\newtheorem{remark}[thm]{Remark}

\allowdisplaybreaks

\title{Matrix-Valued Hermite and Laguerre polynomials 
via Quadratic Transformation}

\author[I. Pacharoni \and V. Torres]{Inés Pacharoni
\and A. Victoria Torres
}
\address{CIEM-FaMAF, Universidad Nacional de Córdoba, Argentina}
\email{ines.pacharoni@unc.edu.ar,victoria.torres.999@unc.edu.ar}

\begin{document}

\begin{abstract}
We present the first systematic extension of the classical Hermite--Laguerre quadratic correspondence to the matrix‑valued setting. 
Starting from a Hermite‑type weight matrix $W(x)=e^{-x^{2}}Z(x)$ with $W(x)=W(-x)$, the change of variables $y=x^{2}$ produces two Laguerre‑type weights with parameters $\alpha=-\tfrac12$ and $\alpha=\tfrac12$ and relates the corresponding sequences of matrix‑valued orthogonal polynomials through an explicit decomposition into even and odd subsequences
We prove that this transformation preserves differential operators and Darboux transformations, thereby establishing a direct structural link between the Hermite and Laguerre sides and providing new constructive tools for the matrix Bochner problem.  
Concrete families---including a new $3\times3$ example and 
 an arbitrary‑size family built from block‑nilpotent matrices 
---illustrate the theory and supply fresh sources of matrix‑valued orthogonal polynomials endowed with non‑trivial differential algebras.


\end{abstract}


\thanks{This paper was partially supported by  CONICET, PIP
N°:11220200102031, SeCyT-UNC}

\maketitle

\section{Introduction}

In the scalar setting, the Hermite and Laguerre polynomials are related through the transformation \( y = x^2 \). Specifically,
\[
h_{2n}(x) = \ell_n^{(-1/2)}(x^2), \qquad h_{2n+1}(x) = x \, \ell_n^{(1/2)}(x^2),
\]
where \( h_n(x) \) denotes the \( n \)-th monic Hermite polynomial, orthogonal with respect to the weight \( w(x) = e^{-x^2} \) on \( \mathbb{R} \), and \( \ell_n^{(\alpha)}(y) \) denotes the \( n \)-th monic Laguerre polynomial, orthogonal with respect to \( w_\alpha(y) = y^\alpha e^{-y} \) on \( (0, \infty) \). These identities relate even and odd Hermite polynomials to Laguerre polynomials with parameters \( \alpha = -\frac{1}{2} \) and \( \alpha = \frac{1}{2} \), respectively.

This relationship, well-established in the scalar case \cite{Szego39}, holds significant importance in mathematical physics and beyond. Hermite polynomials arise as eigenfunctions of the quantum harmonic oscillator, while Laguerre polynomials feature in radial solutions of the Schrödinger equation for hydrogen-like atoms. The transformation \( y = x^2 \) links these contexts, mapping Cartesian coordinates (Hermite) to radial ones (Laguerre), and simplifies computations involving three-term recurrence relations, Rodrigues formulas, and integral representations that appear in quantum mechanics, probability, and approximation theory.

Despite the increasing interest in matrix-valued orthogonal polynomials (MVOP), the classical Hermite–Laguerre relationship has not been explored in this setting. In this paper, we extend it under the transformation \( y = x^2 \), showing that this generalization requires specific structural conditions, (in particular  the symmetry \( W(x) = W(-x) \)), to preserve the key features of the scalar case. We single out a particular family of Hermite-type weights already studied in the literature, now understood within a broader structural framework.

We consider Hermite-type weight matrices of the form
\[
W(x) = e^{-x^2} Z(x), \quad \text{with} \quad W(x) = W(-x),
\]
and show that the change of variables \( y = x^2 \) produces a Laguerre-type weight
\[
V(y) = y^{-1/2} e^{-y} Z(\sqrt{y}),
\]
along with a natural decomposition of the associated monic orthogonal polynomials \( \{ H_n(x) \}_{n\geq 0} \) into even and odd subsequences. More precisely (Theorem~\ref{hermite-laguerre}):
\[
H_{2n}(x) = L_n(x^2), \qquad H_{2n+1}(x) = x F_n(x^2),
\]
where \( \{L_n\}_{n\geq 0}  \) and \( \{F_n\}_{n\geq 0}  \) are the sequences of monic orthogonal polynomials associated with the weights 
$V(y)$ and   $$U(y) = y^{1/2} e^{-y} Z(\sqrt{y}),$$ respectively. These correspond to Laguerre weights with parameters \( \alpha = -\frac{1}{2} \) and \( \alpha = \frac{1}{2} \), extending the scalar case to a matrix-valued framework.

In Section~\ref{sec-operator}, we study how differential operators behave under the transformation \( y = x^2 \). Although any operator \( D \) having the Hermite polynomials as eigenfunctions can be formally transformed into an operator \( \tilde{D} \) via this change of variables, it is not clear a priori whether \( \tilde{D} \) will have the Laguerre-type orthogonal polynomials \( \{L_n\}_{n\geq 0}  \) as eigenfunctions. In Theorem~\ref{symmetry-V}, we show that this is indeed the case when the Hermite-type weight \( W(x) \) satisfies \( W(x) = W(-x) \). In fact, under this assumption, the transformed operator \( \tilde{D} \) belongs to the algebra \( \mathcal{D}(V) \), and moreover, \( W \)-symmetric operators are mapped to \( V \)-symmetric ones. A similar result holds for the weight \( U(y) \), relating \( \mathcal{D}(W) \) to \( \mathcal{D}(U) \) (Theorem~\ref{symmetry-Vhat}). This highlights a structural compatibility between the Hermite and Laguerre settings that fails for general Hermite-type weights. In particular, our result provides a constructive criterion for generating new weights with nontrivial differential algebras, contributing to the broader understanding of the matrix Bochner problem.

In Section~\ref{sec-darboux}, we study how Darboux transformations interact with the Hermite–Laguerre correspondence. We prove that if \( W(x) \) is a Darboux transformation of the classical diagonal Hermite weight \( e^{-x^2} I \), then the associated Laguerre-type weights \( V(y) \) and \( U(y) \) are Darboux transformations of the classical Laguerre weights \( w_{-1/2}(y) I \) and \( w_{1/2}(y) I \), respectively (Theorem~\ref{thm:darboux-equivalence}). This result is illustrated with explicit examples, including a new \( 3 \times 3 \) matrix-valued Hermite-type weight for which we construct the corresponding Laguerre-type orthogonal polynomials and verify the Darboux relations.

Finally, in Section~\ref{sec-NxN}, we extend our results to arbitrary matrix size, showing how the Hermite–Laguerre correspondence and the preservation of operator structures carry over to higher dimensions. These general constructions provide new structural insights into MVOP and practical tools for their analysis.

\smallskip
 Beyond the specific results established in this paper, the matrix-valued Hermite–Laguerre correspondence illustrates a broader principle: classical transformations can retain their algebraic and analytic structure when extended to the matrix setting. We hope this perspective will encourage further exploration of how such transformations can be used to construct and analyze new families of matrix-valued orthogonal polynomials and their differential operators.

\section{Preliminaries on matrix-valued orthogonal polynomials}\label{sec-background}

Let $W(x)$ be an $N \times N$ matrix-valued function defined on an interval $I \subseteq \mathbb{R}$. We say that $W(x)$ is a \emph{weight matrix} if it is integrable on $I$, positive definite almost everywhere, and all its moments exist
\[
\int x^n W(x) \, dx < \infty, \quad \text{for all } n \geq 0.
\]

Given such a weight matrix $W(x)$, we consider the Hermitian sesquilinear form
\[
\langle P, Q \rangle_W = \int P(x) W(x) Q(x)^* \, dx,
\]
on the space of $N \times N$ matrix polynomials $\operatorname{Mat}_N(\mathbb{C})[x]$, where $P(x)^*$ denotes the conjugate transpose of $P(x)$. Throughout this paper, we only consider weight matrices that decay exponentially at infinity, in order to avoid pathological cases and technical complications.

This inner product determines a (unique) sequence of monic orthogonal matrix polynomials $\{P_n\}_{n \geq 0}$ such that $\deg(P_n) = n$. Any other orthogonal sequence $\{Q_n\}_{n \geq 0}$ for the same weight can be written as $Q_n(x) = M_n P_n(x)$ for a sequence of invertible matrices $M_n$.

As in the scalar case, the monic orthogonal polynomials satisfy a three-term recurrence relation
\[
xP_n(x) = P_{n+1}(x) + B_n P_n(x) + A_n P_{n-1}(x), \quad n \geq 0,
\]
with initial condition $P_{-1}(x) = 0$.

Along this paper, we consider that an arbitrary matrix differential operator
\begin{equation}\label{D orden s}
D = \sum_{i=0}^s \frac{d^i}{dx^i} F_i(x),
\end{equation}
acts on the right   on  a matrix-valued function $P$ i.e.
$(P\cdot {D})(x)=\displaystyle \sum_{i=0}^s  \frac{d^i P(x)}{dx^i } F_i(x).$
We consider the algebra of these operators with polynomial coefficients $$\operatorname{Mat}_{N}(\Omega[x])=\Big\{D = \sum_{j=0}^{s} \frac{d^j}{dx^j} F_{j}(x) \, : F_{j} \in \operatorname{Mat}_N(\mathbb{C})[x] \Big \}.$$
We focus on the class of operators whose coefficients $F_i(x)$ are matrix polynomials with $\deg(F_i) \leq i$:
\begin{equation*}
\mathbf{D} = \left\{ \sum_{j=0}^s \frac{d^j}{dx^j} F_j(x) : F_j \in \operatorname{Mat}_N(\mathbb{C})[x], \deg(F_j) \leq j \right\}.
\end{equation*}

From Propositions 2.6 and 2.7 in \cite{GT07} we have the following result 

\begin{prop}
\label{GTeigenvalues}
Let $W(x)$ be a weight matrix and $\{P_n\}_{n \geq 0}$ its monic orthogonal polynomials. If $D$ is a differential operator as in \eqref{D orden s} such that $P_n \cdot D = \Lambda_n P_n$ for all $n\geq 0$, then each $F_i(x)$ is a matrix polynomial of degree at most $i$. Moreover, $D$ is uniquely determined by the sequence $\{\Lambda_n\}_{n \geq 0}$, and the eigenvalues satisfy
\[
\Lambda_n = \sum_{i=0}^s [n]_i F_i^i, \quad \text{for all } n \geq 0,
\]
where $[n]_i = n(n-1)\cdots(n-i+1)$ and $F_i^i$ is the leading coefficient of $F_i(x)$.
\end{prop}

The set of all such operators forms an algebra
\[
\mathcal{D}(W) = \big\{ D \in \operatorname{Mat}_N(\Omega[x]) : P_n \cdot D = \Lambda_n(D) P_n \text{ for all } n \big\},
\]
where the $\Lambda_n(D)$ are constant matrices.

The next proposition, also from \cite{GT07}, provides a useful sufficient condition to recognize operators in $\mathcal{D}(W)$ via symmetry with respect to the inner product:

\begin{prop}[\cite{GT07}]\label{symmetry condition}
Let $W(x)$ be a weight matrix, and let $D = \sum_{i=0}^s \frac{d^i}{dx^i} F_i(x) \in \mathbf D$. 
If
\[
\langle P \cdot D, Q \rangle_W = \langle P, Q \cdot D \rangle_W,  \qquad \text{for all } P,Q \in \operatorname{Mat}_N(\mathbb{C})[x],
\]
then $D \in \mathcal{D}(W)$.
\end{prop}

Operators satisfying this identity are called \emph{$W$-symmetric}. They form a real vector space $\mathcal{S}(W) \subseteq \mathcal{D}(W)$, and one has the decomposition
\[
\mathcal{D}(W) = \mathcal{S}(W) \oplus i \mathcal{S}(W).
\]

\section{ The Hermite-Laguerre relationship}\label{sec-weight}

It is a classical fact that Hermite and Laguerre polynomials are related through the change of variables $y = x^2$. In particular, even and odd Hermite polynomials can be expressed in terms of Laguerre polynomials with different parameters.

The substitution $y = x^2$ transforms the Hermite weight $e^{-x^2} \, dx$ on $\mathbb{R}$ into the Laguerre weight with parameter $\alpha = -\tfrac{1}{2}$ on $(0, \infty)$:
\[
e^{-x^2} \, dx = y^{-1/2} e^{-y} \, dy.
\]

Let $\{h_n(x)\}_{n \geq 0}$ and $\{\ell_n^{(\alpha)}(y)\}_{n \geq 0}$ denote the monic Hermite and Laguerre polynomials of parameter $\alpha$, respectively. Since the even-indexed Hermite polynomials $h_{2n}(x)$ are polynomial functions in $x^2$, the substitution $y = x^2$ yields the identity
\[
h_{2n}(x) = \ell_n^{(-1/2)}(x^2).
\]

The odd-indexed Hermite polynomials $h_{2n+1}(x)$ are odd functions, and they can be written in the form $h_{2n+1}(x) = x \, p_n(x^2)$, where $p_n(y)$ is a monic polynomial of degree $n$. Under the same substitution, one finds that the polynomials $p_n(y)$ are orthogonal with respect to the Laguerre weight $w_{1/2}(y) = y^{1/2} e^{-y}$, leading to
\[
h_{2n+1}(x) = x \, \ell_n^{(1/2)}(x^2).
\]

These identities reflect a classical link between Hermite and Laguerre polynomials under the quadratic transformation $y = x^2$. The goal of this section is to generalize this relation to the setting of matrix-valued orthogonal polynomials.

Throughout this paper, by a \emph{Hermite-type} weight, we mean a weight matrix of the form
\[
W(x) = e^{-x^2} Z(x),
\]
whereas by a \emph{Laguerre-type} weight, we mean a weight matrix of the form
\[
W(x) = x^\alpha e^{-x} Z(x), \quad \text{with } \alpha > -1.
\]
These definitions are not standard but are convenient for our purposes.
In both cases, the weight matrices are assumed to have exponential decay at infinity. Although we do not impose additional explicit restrictions on $Z(x)$, in all explicit examples considered, it will be a matrix polynomial.

We now turn to the matrix-valued setting and investigate how the classical Hermite–Laguerre correspondence extends under the transformation \( y = x^2 \). We focus on Hermite-type weight matrices that are invariant under the reflection \( x \mapsto -x \), that is,
\[
W(x) = e^{-x^2} Z(x), \quad \text{with} \quad W(x) = W(-x).
\]
This structural condition is not satisfied by all known Hermite-type weights, but it plays a fundamental role in our approach: it ensures that the associated orthogonal polynomials exhibit parity behavior analogous to the scalar case, and it allows us to define Laguerre-type weights via the transformation \( y = x^2 \). 

We begin by analyzing how this symmetry impacts the associated orthogonal polynomials. In particular, we show that even and odd degrees correspond to even and odd functions, respectively, mirroring the classical behavior. 

We are grateful to Ignacio Bono Parisi for suggesting the idea behind Proposition~3.1, which led to a significant simplification in the analysis of this parity property.

\begin{prop}\label{par-impar}
Let $W(x) = e^{-x^2} Z(x)$ be a Hermite-type weight matrix satisfying $W(x) = W(-x)$. Then, any sequence $\{P_n(x)\}_{n\geq 0}$ of orthogonal polynomials with respect to $W(x)$ has the property that $P_{2n}(x)$ are even functions, and $P_{2n+1}(x)$ are odd functions.
\end{prop}

\begin{proof}
Define $Q_n(x) = P_n(-x)$. Since $W(x) = W(-x)$, it follows that $\langle Q_n, Q_m \rangle_W = \langle P_n, P_m \rangle_W$. Hence, $\{Q_n(x)\}_{n\geq 0} $ is also orthogonal with respect to $W(x)$, and there exist invertible matrices $M_n$ such that $P_n(-x) = M_n P_n(x)$. Comparing the leading terms, we conclude $P_n(-x) = (-1)^n P_n(x)$, proving the desired parity.
\end{proof}

We now describe explicitly how the transformation $y = x^2$ acts on these weights, giving rise to a Laguerre-type weight of parameter $\alpha=-\tfrac12$.

\begin{prop}\label{inner product}
Under the transformation $y = x^2$, the Hermite-type weight $W(x) = e^{-x^2} Z(x)$, with $W(x)=W(-x)$, transforms into the Laguerre-type weight
\[
V(y) = y^{-1/2} e^{-y} Z(\sqrt{y}), \quad  \qquad y\in (0,\infty).
\]
Moreover, for all matrix-valued functions $F(x)$ and $G(x)$, it holds that
\[
\langle F, G \rangle_W = \langle \tilde{F}, \tilde{G} \rangle_V,
\]
where $\tilde{F}(y) = F(\sqrt{y})$.
\end{prop}
\begin{proof}
Let $ F, G $ be matrix-valued functions. The inner product defined by $ W(x) $ can be written as
\[
\langle F,G \rangle_W=\int_0^\infty F(-x)e^{-x^2} Z(x) G(-x)^* \,dx +\int_0^\infty F(x)e^{-x^2}Z(x)G(x)^* \,dx.
\]

Consider the change of variable $ x = \sqrt{y} $. If both $ F $ and $ G $ are even functions, it follows that
\[
    \langle F, G \rangle_W = \int_0^\infty \tilde{F}(y) y^{-1/2} e^{-y} Z(\sqrt{y}) \tilde{G}(y)^* \, dy = \langle \tilde{F}, \tilde{G} \rangle_V.
\]

The same equality holds if $ F $ and $ G $ are odd functions. Since any function can be written as the sum of an even and an odd function, and $ \langle F, G \rangle_W = 0 $ when $ F $ and $ G $ have different parity, the result follows.
\end{proof}

\smallskip

We now state the main result, establishing explicitly how the Hermite-type matrix polynomials split into even and odd subsequences associated with Laguerre-type matrix polynomials under $y=x^2$.


\begin{thm}\label{hermite-laguerre}
Let $ W(x) = e^{-x^2} Z(x) $ be a Hermite-type weight matrix such that 
$ W(x) = W(-x) $, and let $ \{H_n(x)\}_{n \geq 0} $ be the monic matrix-valued orthogonal polynomials with respect to $ W(x) $. Then the following statements hold:
\begin{enumerate}[i)]
    \item The even-indexed polynomials $ H_{2n}(x) $ satisfy
    \[
    H_{2n}(x) = L_n(x^2), \quad \text{for all } n \geq 0,
    \]
    where $ \{L_n(y)\}_{n \geq 0} $ are the monic orthogonal polynomials for the Laguerre-type weight $ V(y) =  y^{-1/2} e^{-y} Z(\sqrt{y}) $.

    \item The odd-indexed polynomials $ H_{2n+1}(x) $ admit a factorization of the form
    \[
    H_{2n+1}(x) = x \, F_n(x^2), \quad \text{for all } n \geq 0,
    \]
    where the polynomials $ \{F_n(y)\}_{n \geq 0} $ form the monic orthogonal sequence for the Laguerre-type weight $ U(y) = y^{1/2} e^{-y} Z(\sqrt{y})=y V(y) $.
\end{enumerate}
\end{thm}

\begin{proof}
\emph{i) }  From Proposition \ref{par-impar}, we know that $ H_{2n}(x) $ is an even function, and therefore $ G_n(y)= H_{2n}(\sqrt{y}) $ is a monic polynomial of degree $ n $ in the variable $ y $.

Using Proposition \ref{inner product} and the orthogonality of $ H_{2n}(x) $ with respect to $ W(x) $, it follows that $ \{ G_n(y) \}_{n \geq 0} $ is a sequence of monic orthogonal polynomials with respect to the Laguerre-type weight $ V(y) $. 
Hence, 
$G_n(y) = L_n(y)$ for all $n \geq 0$,
and substituting back, we obtain
$H_{2n}(x) = L_n(x^2)$.

\smallskip
\emph{ii) }
Since $ H_{2n+1}(x) $ is an odd polynomial, it can be factored as
$H_{2n+1}(x) = x G_n(x^2),$
where $ G_n(y) $ is a monic polynomial of degree $ n $ in the variable $ y = x^2 $.
Using the orthogonality of $ H_{2n+1}(x) $ with respect to the weight $ W(x) $, and substituting $ y = x^2 $, we compute 
\[
0= \int_{\mathbb{R}} H_{2n+1}(x) W(x) H_{2m+1}(x)^* \, dx=  \int_{0}^\infty G_n(y) \, y V(y) \, G_m(y)^* \, dy.
\]
where $ V(y) = y^{-1/2} W(\sqrt{y}) $ is the Laguerre-type weight obtained under the transformation $ y = x^2 $.
Thus, $ \{G_n(y)\}_{n \geq 0} $ is a sequence of orthogonal polynomials with respect to the weight $ y V(y) = y^{1/2} e^{-y} Z(\sqrt{y}) $.  
Since $ \{F_n(y)\}_{n \geq 0} $ is the unique sequence of monic orthogonal polynomials with respect to $ y V(y) $, and $ G_n(y) $ is also monic, 
it follows that $G_n(y) = F_n(y)$ and we conclude that 
\[
H_{2n+1}(x) = x \, F_n(x^2), \quad \text{for all } n \geq 0.
\]
\end{proof}

\begin{cor}\label{H_L-nomonicos}
Let \( W(x) = e^{-x^2} Z(x) \) be a Hermite-type weight matrix such that  \( W(x) = W(-x) \), and    let \( \{Q_n\}_{n\geq 0} \) be any sequence of MVOP 
with respect to \( W \). Then,
\[
P_n(y) = Q_{2n}(\sqrt{y}) \qquad \text{ and } \qquad R_n(y) = y^{-1/2} Q_{2n+1}(\sqrt{y})
\]
are  sequences of MVOP with respect to the Laguerre-type weights \( V(y) = y^{-1/2} e^{-y} Z(\sqrt{y}) \) and \( U(y) = y^{1/2} e^{-y} Z(\sqrt{y}) \), respectively. 
\end{cor}

\subsection*{Example}
To illustrate clearly the results of this section, we consider the following explicit two-dimensional Hermite-type weight matrix, introduced in \cite{DG2004}:
\begin{equation}\label{W-2x2}
W(x) = e^{-x^{2}} \begin{pmatrix} 1 + a^{2}x^{4} & ax^{2} \\ ax^{2} & 1 \end{pmatrix}, \quad a \neq 0,
\end{equation}
which evidently satisfies the symmetry condition $ W(x) = W(-x) $ required in Theorem \ref{hermite-laguerre}. A sequence of orthogonal polynomials associated to this weight was first provided explicitly in terms of classical Hermite polynomials in Corollary 5.5  of \cite{BPstrong}. From this, one can derive the monic orthogonal polynomials $H_n(x)$ explicitly as
\[
H_{n}(x) = \begin{pmatrix} 
1 & \frac{-a^3n (n-1)(2n+1)}{2\bigl(a^2n(n-1)+4\bigr)}\\[6pt]
0 & \frac{4}{a^2n(n-1)+4}
\end{pmatrix}{\scriptstyle h_n(x)}+ \begin{pmatrix} 
\frac{-2a(2n+1)}{a^2n(n-1)+4} & \frac{2 a^2 (2 n+1) }{a^2n(n-1)+4}\, x^2-1\\[6pt]
\frac{-4}{a^2n(n-1)+4} & \frac{4a}{a^2n(n-1)+4}\, x^2
\end{pmatrix}\tfrac{an(n-1)}4{\scriptstyle h_{n-2}(x)},
\]
where $ h_{n}(x) $ denotes the monic classical Hermite polynomial of degree $n$.

Under the quadratic transformation $ y = x^2 $, the Hermite-type weight $ W(x) $ transforms into the Laguerre-type weight
\begin{equation}\label{V-2x2}
V(y) = y^{-1/2} e^{-y} \begin{pmatrix} 
1 + a^{2}y^{2} & ay \\ 
ay & 1 
\end{pmatrix},
\end{equation}
as described in Proposition \ref{inner product}. From Theorem \ref{hermite-laguerre}, the sequence of monic orthogonal polynomials $\{L_n(y)\}_{n\geq 0}$ associated with the weight $V(y)$ is given by
\[
L_n(y)=H_{2n}(\sqrt{y}),
\]
which can be explicitly written as
\[
L_n(y)=\begin{pmatrix} 
1 & \frac{-a^3n (2n-1)(4n+1)}{2\bigl(a^2n(2n-1)+2\bigr)}\\[6pt]
0 & \frac{2}{a^2n(2n-1)+2}
\end{pmatrix}{\scriptstyle\ell_n^{(-1/2)}(y)}+ \begin{pmatrix} 
\frac{-a(4n+1)}{a^2n(2n-1)+2} & \frac{ a^2 (4n+1)}{a^2n(2n-1)+2} \, y -1\\[6pt]
\frac{-2}{a^2n(2n-1)+2} & \frac{2a }{a^2n(2n-1)+2}\,  y
\end{pmatrix}\tfrac{an(2n-1)}2{\scriptstyle\ell_{n-1}^{(-1/2)}(y)},
\]
where $\ell_n^{(\alpha)}(y)$ denotes the monic classical Laguerre polynomial of parameter $\alpha$.

Finally, consider the Laguerre-type weight
\[
U(y) = y V(y)= y^{1/2} e^{-y}\begin{pmatrix} 
1 + a^{2}y^{2} & ay \\ ay & 1 
\end{pmatrix}.
\]
According to Theorem \ref{hermite-laguerre}, the sequence of monic orthogonal polynomials $\{F_n(y)\}_{n\geq 0}$ for this weight is given by
\[
F_n(y)=y^{-1/2}H_{2n+1}(\sqrt{y}),
\]
which explicitly yields
\[
F_{n}(y) = \begin{pmatrix} 
1 & \frac{-a^3n(2n+1)(4n+3)}{2\bigl(a^2n(2n+1)+2\bigr)}\\[6pt]
0 & \frac{2}{a^2n(2n+1)+2}
\end{pmatrix}{\scriptstyle \ell_n^{(1/2)}(y)}+ \begin{pmatrix} 
\frac{-a(4n+3)}{a^2n(2n+1)+2} & \frac{a^2 (4n+3) }{a^2n(2n+1)+2}\, y-1\\[6pt]
\frac{-2}{a^2n(2n+1)+2} & \frac{2a}{a^2n(2n+1)+2}\, y
\end{pmatrix}\tfrac{an(2n+1)}2{\scriptstyle \ell_{n-1}^{(1/2)}(y)}.
\]

\smallskip

\section{Differential Operators under Quadratic Transformation}
\label{sec-operator}

In this section, we investigate how the algebra $\mathcal{D}(W)$ of differential operators associated with a Hermite-type weight matrix $W(x)$ transforms into the algebra $\mathcal{D}(V)$ corresponding to a Laguerre-type weight matrix $V(y)$ under the change of variables $y = x^2$. The relationship between the even-indexed Hermite polynomials $H_{2n}(x)$ and the Laguerre polynomials $L_n(y)$ plays a central role in this analysis.

Let $W(x) = e^{-x^2} Z(x)$ be a Hermite-type weight matrix satisfying $W(x) = W(-x)$, and let $V(y) = y^{-1/2} e^{-y} Z(\sqrt{y})$ denote the Laguerre-type weight obtained from the change of variable $y = x^2$.

\begin{remark}
Any differential operator 
\[
D = \sum_{j=0}^s \frac{d^j}{dx^j} A_j(x) \in \mathcal{D}(W),
\]
has coefficients $A_{2j}(x)$ that are even functions and $A_{2j+1}(x)$ that are odd functions. This property follows from the fact that the orthogonal polynomials $H_n(x)$ are eigenfunctions of $D$ for all $n$, together with Proposition \ref{par-impar}, which establishes that $H_{2n}(x)$ and $H_{2n+1}(x)$ are even and odd functions, respectively.
\end{remark}

Applying the change of variables $y = x^2$, any operator $D$ is transformed into an operator $\tilde{D}$, explicitly defined by
\begin{equation}\label{def-tildeD}
\bigl(F \cdot D\bigr)(\sqrt{y}) = \bigl(\tilde{F} \cdot \tilde{D}\bigr)(y), \qquad \text{where } \tilde{F}(y) = F(\sqrt{y}).
\end{equation}
For example, a second-order differential operator in $\mathcal{D}(W)$ of the form
\[
D = \frac{d^2}{dx^2} \big( A_2^2 x^2 + A^0_2 \big) + \frac{d}{dx} A_1 x + A_0,
\]
where $A_2^2$, $A_2^0$, $A_1$, and $A_0$ are constant matrices, transforms into the operator
\[
\tilde{D} = 4\frac{d^2}{dy^2} \, y \bigl(A_2^2 y + A_2^0\bigr) + 2\frac{d}{dy} \Bigl( \bigl(A_2^2 + A_1\bigr) y + A_2^0 \Bigr) + A_0.
\]

Note that the structure of the coefficients in $D$ ensures that all coefficients of $\tilde{D}$ are polynomials in $y$.

The following result shows that this transformation maps differential operators in $\mathcal{D}(W)$ into $\mathcal{D}(V)$.

\begin{thm}\label{symmetry-V}
Let $W(x)$ be a Hermite-type weight matrix such that $W(x) = W(-x)$, and let $V(y) = y^{-1/2} W(\sqrt{y})$ be the associated Laguerre-type weight. Then,
\begin{enumerate}[i)]
    \item If $D \in \mathcal{D}(W)$, then the transformed operator $\tilde{D}$ belongs to $\mathcal{D}(V)$.
    \item If $D$ is $W$-symmetric, then $\tilde{D}$ is $V$-symmetric.
\end{enumerate}
\end{thm}

\begin{proof}
Let $H_n(x)$ be the monic orthogonal polynomials associated with $W(x)$. Then
\[
H_n \cdot D = \Lambda_n(D) H_n, \quad \text{for all } n \geq 0,
\]
where $\Lambda_n(D)$ is the matrix eigenvalue. By definition, the transformed function $\tilde{H}_n(y) = H_n(\sqrt{y})$ is an eigenfunction of $\tilde{D}$ with the same eigenvalue.

Theorem \ref{hermite-laguerre} implies that $H_{2n}(x) = L_n(x^2)$, where $L_n(y)$ is the monic orthogonal polynomial associated with $V(y)$. Therefore, $\tilde{H}_{2n}(y) = L_n(y)$ satisfies
\[
L_n \cdot \tilde{D} = \Lambda_{2n}(D) L_n,
\]
proving that $\tilde{D} \in \mathcal{D}(V)$. In particular, $\Lambda_n(\tilde{D}) = \Lambda_{2n}(D)$.

If $D$ is $W$-symmetric, then
\[
\langle P \cdot D, Q \rangle_W = \langle P, Q \cdot D \rangle_W, \quad \text{for all } P, Q \in \operatorname{Mat}(\mathbb{C})[x].
\]
Using Proposition \ref{inner product}, we conclude that the transformed operator $\tilde{D}$ satisfies
\[
\langle \tilde{P} \cdot \tilde{D}, \tilde{Q} \rangle_V = \langle \tilde{P}, \tilde{Q} \cdot \tilde{D} \rangle_V, \quad \text{for all } \tilde{P}, \tilde{Q} \in \operatorname{Mat}_N(\mathbb{C})[y].
\]
Thus, $\tilde{D}$ is $V$-symmetric.
\end{proof}

\begin{remark}\label{remark-poly-coeff}
We have shown that for any $D \in \mathcal{D}(W)$, the transformed operator $\tilde{D}$ has polynomial eigenfunctions of every degree $n$. Therefore, the coefficients of $\tilde{D}$ must be polynomials in $y$ (see Proposition \ref{GTeigenvalues}).

This generalizes the earlier observation for second-order operators, where the polynomial nature of the coefficients was established explicitly.
\end{remark}

In the case where the leading coefficient $A_2(x)$ is the identity matrix, we obtain the following result.

\begin{cor}\label{simetr V-cor}
Let $D = \frac{d^2}{dx^2} I + \frac{d}{dx} A_1 x + A_0$ be a differential operator in $\mathcal{D}(W)$. Then, under the change of variables $y = x^2$, the transformed operator
\[
\tilde{D} = 4 \frac{d^2}{dy^2} y I + 2 \frac{d}{dy} \big(A_1 y + I\big) + A_0,
\]
belongs to $\mathcal{D}(V)$.
\end{cor}

We now examine the connection between odd-indexed Hermite polynomials and the orthogonal polynomials associated with the weight $U(y) = y^{1/2} W(\sqrt{y})$. In this case, $H_{2n+1}(x)$ admits the factorization
\[
H_{2n+1}(x) = x F_n(x^2),
\]
where $F_n(y)$ is a sequence of monic matrix-valued polynomials. Our goal is to deduce a differential equation for $F_n(y)$, starting from the fact that $H_{2n+1}(x)$ is an eigenfunction of a differential operator $D \in \mathcal{D}(W)$.
To isolate the action on $F_n$, we define a new operator $E$ by the identity
\begin{equation}\label{def-E}
E \,x = x\, D  
\end{equation}
and then we define the operator $\tilde{E}$ as the image of $E$ under the change of variables $y = x^2$, following the same rule as in~\eqref{def-tildeD}.

\begin{thm}\label{symmetry-Vhat}
Let $W(x) = e^{-x^2} Z(x)$ be a Hermite-type weight such that $W(x) = W(-x)$, and let $U(y) = y^{1/2} e^{-y} Z(\sqrt{y})$ be the associated Laguerre-type weight with parameter $\alpha = \tfrac{1}{2}$.

If $D \in \mathcal{D}(W)$, then the operator $\tilde{E}$ defined above belongs to $\mathcal{D}(U)$.
\end{thm}

\begin{proof}
Since $D \in \mathcal{D}(W)$, the orthogonal polynomials $H_n(x)$ satisfy
\begin{equation}\label{eigenvalue}
(H_n \cdot D)(x) = \Lambda_n(D) H_n(x), \quad \text{for all } n \geq 0.
\end{equation}

For odd indices, Theorem \ref{hermite-laguerre} gives $H_{2n+1}(x) = x F_n(x^2)$, where $F_n(y)$ is the sequence of monic orthogonal polynomials associated with $U(y)$. Using the identity $x D = E x$ and substituting into~\eqref{eigenvalue}, we obtain
\[
(F_n \cdot E)(x^2) = \Lambda_{2n+1}(D) F_n(x^2).
\]

Applying the change of variables $y = x^2$, and using the definition of $\tilde{E}$, we get
\[
(F_n \cdot \tilde{E})(y) = \Lambda_{2n+1}(D) F_n(y), \quad \text{for all } n \geq 0.
\]
By Proposition \ref{GTeigenvalues}, the operator $\tilde{E}$ has polynomial coefficients, hence $\tilde{E} \in \mathcal{D}(U)$. 
\end{proof}

\subsection*{Example} \mbox{}
To illustrate the results of this section, we consider the Hermite-type weight matrix
\[
W(x) = e^{-x^2}
\begin{pmatrix}
    1 + a^2 x^4 & ab x^4 & a x^2 \\
    ab x^4 & 1 + b^2 x^4 & b x^2 \\
    a x^2 & b x^2 & 1
\end{pmatrix}
= e^{-x^2} e^{A x^2} e^{A^* x^2},
\]
where
\[
A = \begin{psmallmatrix}
    0 & 0 & a \\ 0 & 0 & b \\ 0 & 0 & 0
\end{psmallmatrix},\qquad a^2+b^2>0,.
\]
In \cite{DG2005}, the authors proved that the differential operator
\[
D = \frac{d^2}{dx^2} I + \frac{d}{dx} \, 2x
\begin{pmatrix}
    -1 & 0 & 2a \\
    0 & -1 & 2b \\
    0 & 0 & -1
\end{pmatrix}
+ 
\begin{pmatrix}
    -4 & 0 & 2a \\
    0 & -4 & 2b \\
    0 & 0 & 0
\end{pmatrix},
\]
is a $W$-symmetric differential operator.
Under the change of variable $y = x^2$, the operator $D$ transforms into
\[
\tilde{D} = 4\frac{d^2}{dy^2} I + 2\frac{d}{dy}
\begin{pmatrix}
    1 - 2y & 0 & 4a y \\
    0 & 1 - 2y & 4b y \\
    0 & 0 & 1 - 2y
\end{pmatrix}
+
\begin{pmatrix}
    -4 & 0 & 2a \\
    0 & -4 & 2b \\
    0 & 0 & 0
\end{pmatrix}.
\]
By Theorem~\ref{symmetry-V}, this operator belongs to the algebra $\mathcal{D}(V)$, where
\[
V(y) = y^{-1/2} e^{-y}
\begin{pmatrix}
    1 + a^2 y^2 & ab y^2 & a y \\
    ab y^2 & 1 + b^2 y^2 & b y \\
    a y & b y & 1
\end{pmatrix},
\]
is a Laguerre-type weight with parameter $\alpha = -\tfrac12$. Moreover, $\tilde{D}$ is symmetric with respect to $V(y)$.

Similarly, Theorem \ref{symmetry-Vhat} implies that the differential operator
\[
\tilde{E} = 4\frac{d^2}{dy^2} I + 2\frac{d}{dy}
\begin{pmatrix}
    3 - 2y & 0 & 4a y \\
    0 & 3 - 2y & 4b y \\
    0 & 0 & 3 - 2y
\end{pmatrix}
+
\begin{pmatrix}
    -6 & 0 & 6a \\
    0 & -6 & 6b \\
    0 & 0 & -2
\end{pmatrix},
\]
belongs to $\mathcal{D}(U)$, where
\[
U(y) = y V(y) = y^{1/2} e^{-y}
\begin{pmatrix}
    1 + a^2 y^2 & ab y^2 & a y \\
    ab y^2 & 1 + b^2 y^2 & b y \\
    a y & b y & 1
\end{pmatrix}.
\]

\section{Darboux Transformations and Classical Weights}\label{sec-darboux}

Darboux transformations are a powerful tool for generating new weight matrices from classical ones, while preserving the existence of differential operators in their associated algebras. Originally developed in the scalar case for orthogonal polynomials and differential equations \cites{MS91,Zhed97}, they have been extended to the matrix setting, where they play a key role in understanding spectral properties and bispectrality \cites{Gru96,Gru10,GY02}. Their relevance has become particularly clear in relation to the matrix Bochner problem: in \cite{CY18}, Casper and Yakimov show that, under certain conditions, all weight matrices admitting a second-order differential operator in the algebra \( \mathcal{D}(W) \)---that is, all solutions to the matrix Bochner problem---can be obtained as Darboux transformations of classical weights. In this section, we examine how the Hermite–Laguerre correspondence established in Section~\ref{sec-weight} interacts with Darboux transformations. Specifically, we prove that if a symmetric Hermite-type weight is a Darboux transformation of the classical diagonal weight \( w(x) = e^{-x^2} I \), then the associated Laguerre-type weights also inherit this structure. This not only extends the scope of the Hermite–Laguerre relationship to a broader algebraic framework, but also contributes to the ongoing effort to characterize which Bochner-type solutions arise from Darboux transformations. Indeed, as shown in \cites{BP23,BP24-2}, there exist Bochner weights that do not arise in this way, highlighting the interest of understanding such transformations in detail.

To set the stage for our results, we now review the definition and fundamental properties of Darboux transformations in the matrix-valued setting.

We say that  $\mathcal{V} \in \operatorname{Mat}_{N}(\Omega [x])$ is a  {\em degree-preserving} differential operator if the degree of $P(x) \cdot \mathcal{V}$ is equal to the degree of $P(x)$ for all $P(x) \in \operatorname{Mat}_{N}(\mathbb{C})[x]$.


\begin{definition}\label{Darboux-def}
Let $W(x)$ and $\tilde{W}(x)$ be weight matrices with associated monic orthogonal polynomials $\{P_n\}_{n \geq 0}$ and $\{\tilde{P}_n\}_{n \geq 0}$, respectively. We say that $\tilde{W}(x)$ is a \emph{Darboux transformation} of $W(x)$ if there exists an operator $D \in \mathcal{D}(W)$ that admits a factorization $D = \mathcal{V} \mathcal{N}$, where $\mathcal{V}$ and $\mathcal{N}$ are degree-preserving operators and
\[
P_n \cdot \mathcal{V} = A_n \tilde{P}_n,
\]
for a sequence of  matrices $A_n \in \operatorname{Mat}_N(\mathbb{C})$.
\end{definition}

As a consequence of the above definition,  we have that 
the operator $\tilde D= \mathcal N \mathcal V\in \mathcal D(\tilde W)$. Moreover, we have 
\[
\mathcal{V} \mathcal{D}(\tilde{W}) \mathcal{N} \subseteq \mathcal{D}(W)
\quad \text{and} \quad
\mathcal{N} \mathcal{D}(W) \mathcal{V} \subseteq \mathcal{D}(\tilde{W}).
\]

The following result provides a practical characterization of Darboux transformations in terms of a single intertwining operator,
without requiring the factorization of a differential operator.

\begin{thm}[\cite{BP25}]\label{darboux-thm}
A weight 
$\tilde W(x)$ is a {Darboux transformation} of $W(x)$ if and only if 
there exists a differential operator $\mathcal{V}\in \operatorname{Mat}_{N}(\Omega[x])$ such that 
    \begin{equation}\label{darboux-1}
            P_{n} \cdot \mathcal{V}  = A_{n}\tilde {P}_{n}, \quad \text{ for all } n\in \mathbb{N}_{0}, 
    \end{equation}
where $A_{n} \in \operatorname{Mat}_{N}(\mathbb{C})$ invertible for all 
$n$.
\end{thm}

With this general framework in place, we now turn to Hermite-type weights and explore how the Darboux property is preserved under the transformation $y = x^2$. This leads to our main result in this section, which shows that if a Hermite-type weight arises as a Darboux transformation of the classical scalar weight, then the associated Laguerre-type weights inherit this structure.

Let us consider a Hermite-type weight 
\[
W(x) = e^{-x^2} Z(x), \quad \text{such that } W(x) = W(-x),
\]
and denote by \( H_n(x) \) the monic orthogonal polynomials associated with \( W(x) \), and by \( h_n(x) \) the classical monic Hermite polynomials orthogonal with respect to \( e^{-x^2} \).

\begin{thm}\label{thm:darboux-equivalence}
If \( W(x) = e^{-x^2} Z(x) \) is a Darboux transformation of \( w(x) = e^{-x^2} I \), then the Laguerre-type weights 
\[
V(y) = y^{-1/2} e^{-y} Z(\sqrt{y}) \qquad \text{and} \qquad U(y) = y^{1/2} e^{-y} Z(\sqrt{y}),
\]
are Darboux transformations of the diagonal Laguerre weights \( v(y) = y^{-1/2} e^{-y} I \) and \( u(y) = y^{1/2} e^{-y} I \), respectively.
\end{thm}

\begin{proof}
By Theorem \ref{darboux-thm}, there exists a differential operator \( \mathcal{V} \) with polynomial coefficients such that 
\begin{equation}\label{Darboux-herm}
(h_n I\cdot \mathcal{V})(x) = A_n H_n(x), \quad \text{for all } n \geq 0,
\end{equation}
where \( A_n \in \operatorname{Mat}_N(\mathbb{C}) \) is invertible for all \( n \). Under the change of variables \( y = x^2 \), this operator transforms into \( \tilde{\mathcal{V}} \), as defined in Section \ref{sec-operator}, yielding:
\[
(\tilde{h}_n I\cdot \tilde{\mathcal{V}})(y) = A_n \tilde{H}_n(y).
\]

According to the results
in Theorem \ref{hermite-laguerre}, we have
\[
H_{2n}(x) = L_n(x^2) \quad \text{and} \quad H_{2n+1}(x) = x F_n(x^2),
\]
where \( L_n(y) \) and \( F_n(y) \) are the monic orthogonal polynomials associated with \( V(y) \) and \( U(y) \), respectively. Similarly, for the classical scalar Hermite polynomials, we have:
\[
h_{2n}(x) = \ell_n^{(-1/2)}(x^2) \quad \text{and} \quad h_{2n+1}(x) = x \ell_n^{(1/2)}(x^2),
\]
where \( \ell_n^{(\alpha)}(y) \) denotes the monic Laguerre polynomials orthogonal with respect to \( y^{\alpha}e^{-y} \). Substituting into \eqref{Darboux-herm}, we obtain
\begin{equation}\label{Darboux-Laguerre}
(\ell_n^{(-1/2)}I \cdot \tilde{\mathcal{V}})(y) = A_{2n} L_n(y) \quad \text{and} \quad (\ell_n^{(1/2)}I \cdot \tilde{\mathcal{E}})(y) = A_{2n+1} F_n(y),
\end{equation}
where \( \tilde{\mathcal{E}} \) is the operator obtained by applying the transformation \( y = x^2 \) to the operator $\mathcal E$ defined by  \( \mathcal{E} \,x= x \mathcal{V}  \).

The coefficients of \( \tilde{\mathcal{V}} \) and \( \tilde{\mathcal{E}} \) are polynomials in \( y \). Indeed, these operators map the sequence \( \{\ell_n^{(\alpha)}(y)\}_{n \geq 0} \) into sequences of polynomials, as shown in \eqref{Darboux-Laguerre}. An inductive argument on the degree of the resulting polynomials implies that their coefficients must be polynomials in \( y \). This is consistent with \( \mathcal{V} \) having polynomial coefficients, which, due to the symmetry \( W(x) = W(-x) \), are even or odd functions of \( x \), ensuring that the transformation \( y = x^2 \) yields no non-polynomial terms.

Finally, by Theorem \ref{darboux-thm}, the identities in \eqref{Darboux-Laguerre} imply that \( V(y) \) and \( U(y) \) are Darboux transformations of \( v(y) \) and \( u(y) \), respectively.
\end{proof}

This result shows that both Laguerre-type weights arising from the Hermite–Laguerre transformation inherit the Darboux structure from their Hermite counterparts under the quadratic change \( y = x^2 \). The following examples illustrate this correspondence with explicit computations.




\subsection*{Example 1} 
It is known that the weight \( W(x) \), given in \eqref{W-2x2}, is a Darboux transformation of the diagonal Hermite weight \( w(x) = e^{-x^2} I \); see \cite{BPstrong}. Then, as a consequence of Theorem \ref{thm:darboux-equivalence}, we conclude that 
\[
V(y) = y^{-1/2} e^{-y} \begin{pmatrix} 
1 + a^{2}y^{2} & ay \\ 
ay & 1 
\end{pmatrix}, \qquad \text{and} \qquad U(y) = y^{1/2} e^{-y} \begin{pmatrix} 
1 + a^{2}y^{2} & ay \\ 
ay & 1 
\end{pmatrix}
\]
can be obtained as Darboux transformations of the classical Laguerre weights \( v(y) = y^{-1/2} e^{-y} I \) and \( u(y) = y^{1/2} e^{-y} I \), respectively.

\subsection*{Example 2}

We next consider a new example of a Hermite-type weight within the scope of Theorem~\ref{thm:darboux-equivalence}. Let 
\[
W(x) = e^{-x^2} \begin{pmatrix} 
1 + a^2 x^4 & a x^2 & a b x^4 \\ 
a x^2 & 1 & b x^2 \\ 
a b x^4 & b x^2 & 1 + b^2 x^4 
\end{pmatrix},
\]
which is factorized as \( W(x) = e^{-x^2} e^{A x^2} e^{A^* x^2} \), with
\[
A = \begin{pmatrix}
0 & a & 0 \\ 0 & 0 & 0 \\ 0 & b & 0
\end{pmatrix}, \quad a^2 + b^2 > 0.
\]

Based on the framework developed, for example, in \cites{BP23b, BPstrong}, we explicitly computed the differential operators \( \mathcal{V} \) and \( \mathcal{N} \) involved in Definition~\ref{Darboux-def}, showing that \( W(x) \) is a Darboux transformation of the classical weight \( w(x) = e^{-x^2} I \). While we omit some of the technical details for brevity, we present the explicit expressions of \( \mathcal{V} \) and \( \mathcal{N} \), which fully reveal the structure of the Darboux transformation.
 These operators are

\begin{align*}
\mathcal{V} &= \frac{d^2}{dx^2} \begin{pmatrix}
-a^2 - b^2 & a (a^2 + b^2) x^2 & 0 \\
0 & b (a^2 + b^2) x^2 & -a^2 - b^2 \\
0 & -\frac{(a^2 + b^2)^2}{4} & 0
\end{pmatrix} 
+ \frac{d}{dx} \begin{pmatrix}
2 b^2 x & 0 & -2 a b x \\
-2 a b x & 0 & 2 a^2 x \\
0 & (a^2 + b^2)^2 x & 0
\end{pmatrix} \\ &\quad  + \begin{pmatrix}
0 & 4 a & 0 \\
0 & 4 b & 0 \\
-a (a^2 + b^2) & \frac{(a^2 + b^2)^2}{2} & -b (a^2 + b^2)
\end{pmatrix},
\intertext{and}
\mathcal{N} &= \frac{d^2}{dx^2} \begin{pmatrix}
-a^2 - b^2 & 0 & -\frac{a (a^2 + b^2)^2 x^2}{4} \\
0 & 0 & -\frac{(a^2 + b^2)^2}{4} \\
0 & -a^2 - b^2 & -\frac{b (a^2 + b^2)^2 x^2}{4}
\end{pmatrix} \\
&\quad  + \frac{d}{dx} \begin{pmatrix}
2 (2 a^2 + b^2) x & 2 a b x & -a (a^2 + b^2)^2 x \\
0 & 0 & 0 \\
2 a b x & 2 (a^2 + 2 b^2) x & -b (a^2 + b^2)^2 x
\end{pmatrix} 
+ \begin{pmatrix}
2 a^2 & 2 a b & -\frac{a (a^2 + b^2)(2 + a^2 + b^2)}{2} \\
4 a & 4 b & 0 \\
2 a b & 2 b^2 & -\frac{b (a^2 + b^2)(2 + a^2 + b^2)}{2}
\end{pmatrix}.
\end{align*}

The composition \( D = \mathcal{V} \mathcal{N} \) defines a fourth-order differential operator in the algebra \( \mathcal{D}(w) \), since each of its entries can be expressed as a polynomial in \( \delta = \frac{d^2}{dx^2} - 2 \frac{d}{dx} x\), the classical second-order Hermite operator.
In fact, we have 

\begin{align*}  
D &= \begin{psmallmatrix}
(a^2 + b^2) \bigl( (a^2 + b^2) \delta^2 + 2 a^2 \delta \bigr) + 16 a^2 & 2 a b \bigl( (a^2 + b^2) \delta + 8 \bigr) & 0 \\
2 a b \bigl( (a^2 + b^2) \delta + 8 \bigr) & (a^2 + b^2) \bigl( (a^2 + b^2) \delta^2 + 2 b^2 \delta \bigr) + 16 b^2 & 0 \\
0 & 0 & (a^2 + b^2) \bigl( (a^2 + b^2) (\delta^2 - 6 \delta + 8) + 16 \bigr)
\end{psmallmatrix}.
\end{align*}

\medskip

By Theorem \ref{darboux-thm}, the sequence \( Q_n(x) = h_n I \cdot \mathcal{V} \), where \( h_n(x) \) denotes the \( n \)-th monic Hermite polynomial, gives orthogonal polynomials with respect to \( W(x) \), not necessarily monic. Normalizing \( \{Q_n(x)\}_{n \geq 0} \) yields the monic sequence \( \{H_n(x)\}_{n \geq 0} \):
\begin{equation}\label{Hermite3x3}
Q_n(x) = \begin{psmallmatrix}
2 b^2 n & 4 a & -2 a b n \\
-2 a b n & 4 b & 2 a^2 n \\
-a (a^2 + b^2) & \frac{(a^2 + b^2)^2 (2 n+1)}{2} & -b (a^2 + b^2)
\end{psmallmatrix} h_n(x) +
\begin{psmallmatrix}
-a^2 & a (a^2 + b^2) x^2 & -a b \\
- a b & b (a^2 + b^2) x^2 & -b^2 \\
0 & \frac{(a^2 + b^2)^2}{4} & 0
\end{psmallmatrix} n(n-1) h_{n-2}(x).
\end{equation}

\ 

Applying the quadratic change of variables \( y = x^2 \), we obtain the corresponding operators in the Laguerre setting. In particular, the weight 
\[
V(y) = y^{-1/2} e^{-y} \begin{pmatrix} 
1 + a^2 y^2 & a y & a b y^2 \\ 
a y & 1 & b y \\ 
a b y^2 & b y & 1 + b^2 y^2 
\end{pmatrix},
\]
is a Darboux transformation of the classical diagonal Laguerre weight $$ v(y) = y^{-1/2} e^{-y} I ,$$  and the degree-preserving operators \( \mathcal{V} \) and \( \mathcal{N} \) transform accordingly under the change of variables, yielding
\[
\tilde D = \tilde{\mathcal{V}} \tilde{\mathcal{N}} \in \mathcal{D}(v),
\]
where \( \tilde{\mathcal{V}} \) and \( \tilde{\mathcal{N}} \) are the images of \( \mathcal{V} \) and \( \mathcal{N} \), respectively, under the substitution \( y = x^2 \).
Explicit expressions for \( \tilde{\mathcal{V}} \) and \( \tilde{\mathcal{N}} \) are given below:

\begin{align*}
\tilde{\mathcal{V}} &= (a^2 + b^2) \frac{d^2}{dy^2} \begin{pmatrix}
-4 y & 4 a y^2 & 0 \\
0 & 4 b y^2 & -4 y \\
0 & -(a^2 + b^2) y & 0
\end{pmatrix} \displaybreak[0]\\ 
& \quad 
+ \frac{d}{dy} \begin{pmatrix}
4 b^2 y - 2 (a^2 + b^2) & 2 a (a^2 + b^2) y & -4 a b y \\
-4 a b y & 2 b (a^2 + b^2) y & 4 a^2 y - 2 (a^2 + b^2) \\
0 & (a^2 + b^2)^2 (2 y - 1/2) & 0
\end{pmatrix} \displaybreak[0]
\\ &\quad 
+ \begin{pmatrix}
0 & 4 a & 0 \\
0 & 4 b & 0 \\
-a (a^2 + b^2) & \frac{(a^2 + b^2)^2}{2} & -b (a^2 + b^2)
\end{pmatrix},
\end{align*}

and

\begin{align*}
\tilde{\mathcal{N}} &= -(a^2 + b^2) \frac{d^2}{dy^2} \begin{pmatrix}
4 y & 0 & a (a^2 + b^2) y^2 \\
0 & 0 & (a^2 + b^2) y \\
0 & 4 y & b (a^2 + b^2) y^2
\end{pmatrix} \\ & \quad 
+ \frac{d}{dy} \begin{pmatrix}
2 b^2 (2 y - 1) + 2 a^2 (4 y - 1) & 4 a b y & -\frac{5 a (a^2 + b^2)^2 y}{2} \\
0 & 0 & -\frac{(a^2 + b^2)^2}{2} \\
4 a b y & 2 a^2 (2 y - 1) + 2 b^2 (4 y - 1) & -\frac{5 b (a^2 + b^2)^2 y}{2}
\end{pmatrix} \\
&\quad + \begin{pmatrix}
2 a^2 & 2 a b & -\frac{a (a^2 + b^2)(2 + a^2 + b^2)}{2} \\
4 a & 4 b & 0 \\
2 a b & 2 b^2 & -\frac{b (a^2 + b^2)(2 + a^2 + b^2)}{2}
\end{pmatrix}.
\end{align*}

Let us observe that these operators satisfy that the fourth-order differential operator \( \tilde{\mathcal{N}} \tilde{\mathcal{V}} \) lies in \( \mathcal{D}(V) \).

\medskip
By Theorem \ref{hermite-laguerre}, the monic orthogonal polynomials with respect to \( V(y) \) are given by \( L_n(y) = H_{2n}(\sqrt{y}) \). Explicitly,
\begin{align*}
L_n(y) &= \begin{pmatrix}
1 & -\frac{a (a^2 + b^2) n (-1 + 2 n) (1 + 4 n)}{4 + 2 (a^2 + b^2) n (-1 + 2 n)} & 0 \\
0 & \frac{2}{2 + (a^2 + b^2) n (-1 + 2 n)} & 0 \\
0 & -\frac{b (a^2 + b^2) n (-1 + 2 n) (1 + 4 n)}{4 + 2 (a^2 + b^2) n (-1 + 2 n)} & 1
\end{pmatrix} \ell_n^{(-1/2)}(y) \\
&\quad + \begin{pmatrix}
\frac{a^2 (1 + 4 n)}{2} & a + \frac{a (a^2 + b^2) (n (-1 + 2 n - 4 y) - y)}{2} & \frac{a b (1 + 4 n)}{2} \\
a & -(a^2 + b^2) y & b \\
\frac{a b (1 + 4 n)}{2} & b + \frac{b (a^2 + b^2) (n (-1 + 2 n - 4 y) - y)}{2} & \frac{b^2 (1 + 4 n)}{2}
\end{pmatrix} \frac{n (1 - 2 n)}{2 + (a^2 + b^2) n (-1 + 2 n)} \ell_{n-1}^{(-1/2)}(y).
\end{align*}

\smallskip

In addition to \( V(y) \), we also consider the Laguerre-type weight
\[
U(y) = y^{1/2} e^{-y} \begin{pmatrix} 
1 + a^2 y^2 & a y & a b y^2 \\ 
a y & 1 & b y \\ 
a b y^2 & b y & 1 + b^2 y^2 
\end{pmatrix},
\]
which arises from the odd-degree Hermite polynomials, as described in Theorem~\ref{hermite-laguerre}.

A sequence of orthogonal polynomials for \( U(y) \) is given by \( P_n(y) = y^{-1/2} Q_{2n+1}(\sqrt{y}) \), where \( Q_n(x) \) is defined in \eqref{Hermite3x3} (see Corollary~\ref{H_L-nomonicos}). Explicitly,
\[
P_n(y) = \begin{psmallmatrix}
2 b^2 (1 + 2 n) & 4 a & -2 a b (1 + 2 n) \\
-2 a b (1 + 2 n) & 4 b & 2 a^2 (1 + 2 n) \\
-a (a^2 + b^2) & \frac{(a^2 + b^2)^2 (3 + 4 n)}{2} & -b (a^2 + b^2)
\end{psmallmatrix} \ell_n^{(1/2)}(y) + \begin{psmallmatrix}
-2 a^2 & 2 a (a^2 + b^2) y & -2 a b \\
-2 a b & 2 b (a^2 + b^2) y & -2 b^2 \\
0 & \frac{(a^2 + b^2)^2}{2} & 0
\end{psmallmatrix} n (2 n + 1) \ell_{n-1}^{(1/2)}(y),
\]
where \( \ell_n^{(1/2)}(y) \) denotes the classical monic Laguerre polynomial with parameter \( \alpha = \frac{1}{2} \).

\smallskip
From Theorem~\ref{thm:darboux-equivalence}, the weight \( U(y) \) is a Darboux transformation of the classical diagonal Laguerre weight \( u(y) = y^{1/2} e^{-y} I \). The differential operator \( \tilde{\mathcal{E}} \), obtained from the operator \( \mathcal{E} \) defined by \( \mathcal{E} x = x \mathcal{V} \) via the change of variables \( y = x^2 \), satisfies
\[
(\ell_n^{(1/2)} \cdot \tilde{\mathcal{E}})(y) = A_{2n+1} F_n(y),
\]
where \( F_n(y) \) is the monic orthogonal polynomial associated with \( U(y) \) as in Theorem~\ref{hermite-laguerre}. The operator \( \tilde{\mathcal{E}} \) is given explicitly by:
\begin{align*}
\tilde{\mathcal{E}} &= (a^2 + b^2) \frac{d^2}{dy^2} \begin{pmatrix}
-4 y & 4 a y^2 & 0 \\
0 & 4 b y^2 & -4 y \\
0 & -(a^2 + b^2) y & 0
\end{pmatrix} \\ &\quad + 2\frac{d}{dy} \begin{pmatrix}
2b^2y-3(a^2+b^2)& 3 a (a^2 + b^2) y & -2 a b y \\
-2 a b y & 3 b (a^2 + b^2) y & 2a^2y-3(a^2+b^2) \\
0 & \frac{(a^2 + b^2)^2 (4 y - 3)}{2} & 0
\end{pmatrix} \\
&\quad + \begin{pmatrix}
2 b^2 & 4 a & -2 a b \\
-2 a b & 4 b & 2 a^2 \\
-a (a^2 + b^2) & \frac{3 (a^2 + b^2)^2}{2} & -b (a^2 + b^2)
\end{pmatrix}.
\end{align*}

The operator
\begin{align*}
\tilde{\mathcal{F}} &= -(a^2 + b^2) \frac{d^2}{dy^2} \begin{pmatrix}
4 y & 0 & a (a^2 + b^2) y^2 \\
0 & 0 & (a^2 + b^2) y \\
0 & 4 y & b (a^2 + b^2) y^2
\end{pmatrix} \\ 
&\quad + \frac{d}{dy} \begin{pmatrix}
2 b^2 (2 y - 3) + 2 a^2 (4 y - 3) & 4 a b y & -\frac{7 a (a^2 + b^2)^2 y}{2} \\
0 & 0 & -\frac{3(a^2 + b^2)^2}{2} \\
4 a b y & 2 a^2 (2 y - 3) + 2 b^2 (4 y - 3) & -\frac{7 b (a^2 + b^2)^2 y}{2}
\end{pmatrix} \\
&\quad + \begin{pmatrix}
2 (3 a^2 + b^2) & 4 a b & -\frac{a (a^2 + b^2)(2 + 3 a^2 + 3 b^2)}{2} \\
4 a & 4 b & 0 \\
4 a b & 2 (a^2 + 3 b^2) & -\frac{b (a^2 + b^2)(2 + 3 a^2 + 3 b^2)}{2}
\end{pmatrix}
\end{align*}
satisfies  \( \tilde{\mathcal{E}} \tilde{\mathcal{F}} \in \mathcal{D}(u) \) and \( \tilde{\mathcal{F}} \tilde{\mathcal{E}} \in \mathcal{D}(U) \), 
closing the construction on the Laguerre side.
\smallskip

This example confirms the compatibility of the Darboux framework with the Hermite–Laguerre correspondence, while explicitly constructing the associated operators on the Laguerre side and illustrating the broader applicability of our method in the matrix-valued setting.

\section{Some examples in arbitrary dimension}
\label{sec-NxN}
In this section, we present a broad family of Hermite-type weight matrices \( W(x) \) in arbitrary dimension \( N \). Originally introduced in \cite{DG2005}, this family is particularly relevant for our purposes because it satisfies the symmetry condition \( W(x) = W(-x) \), which is not a generic property among Hermite-type weights. Although both the weights and their associated second-order differential operators are already known, we revisit them here to illustrate how the general results developed in this paper apply concretely in a highly structured setting.


The weights are defined by
\[
W(x) = e^{-x^2} e^{B x^2} e^{B^* x^2},
\]
where $ B $ is a block-nilpotent matrix given by:
\begin{equation}\label{ex-B}
    B = \begin{psmallmatrix}
        0 & V_1 & -V_1 V_2 & V_1 V_2 V_3 & \cdots & (-1)^k V_1 V_2 \cdots V_{k-1} \\
        0 & 0 & V_2 & -V_2 V_3 & \cdots & (-1)^{k-1} V_2 \cdots V_{k-1} \\
        0 & 0 & 0 & V_3 & \cdots & (-1)^{k-2} V_3 \cdots V_{k-1} \\
        \vdots & \vdots & \vdots & \vdots & \ddots & \vdots \\
        0 & 0 & 0 & 0 & \cdots & -V_{k-2} V_{k-1} \\
        0 & 0 & 0 & 0 & \cdots & V_{k-1} \\
        0 & 0 & 0 & 0 & \cdots & 0
    \end{psmallmatrix},
\end{equation}
where $ N = n_1 + \cdots + n_k $, with $ 1 \leq k \leq N $, and each $ V_i $ is a nonzero matrix of size $ n_i \times n_{i+1} $, for $ i = 1, \dots, k-1 $. 

In the particular case $N = 2$, the matrix $B$ reduces to $B = \begin{psmallmatrix} 0 & a \\ 0 & 0 \end{psmallmatrix}$, and the weight $W(x)$ coincides with the example given at the end of Section~\ref{sec-weight}, where it is written explicitly as a $2 \times 2$ matrix-valued function.
The example in Section \ref{sec-operator} corresponds to $N=3$ and $k=2$.

All these weights admit a symmetric second-order differential operator in the algebra $ \mathcal{D}(W) $:
\[
D = \frac{d^2}{dx^2} I + \frac{d}{dx} \, 2x(2B - I) + A_0,
\]
where $ A_0 $ is given explicitly by
\begin{equation}\label{ex-A0}
    A_0 = 2B - 4 \operatorname{diag}\bigl((k-1)I_{n_1},\, (k-2)I_{n_2},\, \ldots, \, I_{n_{k-1}}, \, 0 \cdot I_{n_k}\bigr).
\end{equation}

Since $ W(x) = W(-x) $, the change of variables $y = x^2$ transforms $W(x)$ into a Laguerre-type weight with parameter $ \alpha = -\tfrac12 $:
\[
V(y) = y^{-1/2} W(\sqrt{y}) = y^{-1/2} e^{-y} e^{B y} e^{B^* y}.
\]
By Theorem~\ref{symmetry-V} and Corollary~\ref{simetr V-cor}, the differential operator $ D \in \mathcal{D}(W) $ transforms into the Laguerre-type operator
\[
\tilde{D} = 4\frac{d^2}{dy^2} y I + 2\frac{d}{dy}(4B - I) + A_0,
\]
which is a symmetric operator in $ \mathcal{D}(V) $.

By factoring $ x $ from the odd-degree Hermite polynomials $ H_{2n+1}(x) $, we obtain the Laguerre-type weight
\[
U(y) = y^{1/2} W(\sqrt{y}) = y^{1/2} e^{-y} e^{B y} e^{B^* y},
\]
corresponding to the parameter $ \alpha = \tfrac12 $. By Theorem~\ref{symmetry-Vhat}, the differential operator $ D \in \mathcal{D}(W) $ transforms into the Laguerre-type operator
\[
\tilde{E} = 4\frac{d^2}{dy^2} y I + 2\frac{d}{dy} \bigl(2y(2B - I) + I\bigr) + A_0,
\]
which belongs to $ \mathcal{D}(U) $.

\medskip
This family provides an explicit illustration of the Hermite–Laguerre correspondence established in Sections~\ref{sec-weight} and \ref{sec-operator}, showing how the weights and associated differential operators transform under the change of variables \( y = x^2 \).
It reinforces the general principles developed in this paper and demonstrates their applicability to structured families of weights beyond low-dimensional cases.

\begin{remark}
To the best of our knowledge, the family of Hermite-type weights presented here is the only example currently available in the literature that explicitly satisfies the symmetry condition \( W(x) = W(-x) \) and admits a second-order differential operator in its algebra \( \mathcal{D}(W) \). While matrix-valued Laguerre-type weights with general values of the parameter \( \alpha \) were studied in \cite{DG2005-b}, our construction yields only the specific cases \( \alpha = \pm 1/2 \), which arise naturally from the Hermite–Laguerre correspondence under the quadratic transformation. This highlights the structural significance of the symmetry condition, which is essential for extending the classical relation to the matrix-valued setting. Whether this symmetric family can be obtained via Darboux transformations from classical weights remains an open question, and we conjecture this to be the case.
\end{remark}

\end{document}